\documentclass[a4paper,11pt]{article}
\usepackage[utf8x]{inputenc}
\usepackage[left=20mm, right=20mm, top=20mm, bottom=20mm]{geometry}
\usepackage{color, hyperref}
\usepackage{amsmath}
\usepackage{tikz-cd}
\usepackage{amssymb}
\usepackage{amsfonts}
\usepackage{booktabs}
\usepackage{amsthm,amscd}
\usepackage{authblk}
\usepackage{tabularx}
\newtheorem{theorem}{Theorem}[section]
\newtheorem{lemma}[theorem]{Lemma}

\newtheorem{proposition}[theorem]{Proposition}
\theoremstyle{definition}
\newtheorem{definition}[theorem]{Definition}
\newtheorem{example}[theorem]{Example}

\theoremstyle{remark}
\newtheorem{remark}[theorem]{Remark}

\numberwithin{equation}{section}

\title{Metrizability and Dynamics of Weil Bundles}

\author[1]{Stephane Tchuiaga \thanks{tchuiagas@gmail.com}}
\author[2]{Moussa Koivogui \thanks{moussa.koivogui@esatic.ci}}
\author[3]{Fidèle Balibuno \thanks{fidele.balibuno@unikin.ac}}
\affil[1]{Department of Mathematics, University of Buea, South West Region, Cameroon}
\affil[2]{Ecole Superieure Africaine des Technologies de l'Information et de Communication,
	C\^ote d’ Ivoire}
\affil[3]{Department of Mathematics and Computer Science, Faculty of Sciences and Technologies\\
	University of Kinshasa, Kinshasa, D.R.Congo}

\date{ }

\begin{document}

		\maketitle

	\begin{abstract} 
This paper bridges synthetic and classical differential geometry by investigating the metrizability and dynamics of Weil bundles. For a smooth, compact manifold \(M\) and a Weil algebra \(\mathbf{A}\), we prove that the manifold \(M^\mathbf{A}\) of \(\mathbf{A}\)-points admits a canonical, weighted metric \(\mathfrak{d}_w\) that encodes both base-manifold geometry and infinitesimal deformations. Our approach relies on constructions and  methods of local and global analysis. Key results include: (1). Metrization: \(\mathfrak{d}_w\) induces a complete metric topology on \(M^\mathbf{A}\). (2). Path Lifting: Curves lift from \(M\) to \(M^\mathbf{A}\) while preserving topological invariants.
(3). Dynamics: Fixed-point theorems for diffeomorphisms on \(M^\mathbf{A}\) connected to stability analysis.
(4). Topological Equivalence: \(H^*(M^\mathbf{A}) \cong H^*(M)\) and \(\pi_\ast(M^\mathbf{A}) \cong \pi_\ast(M)\).  
	\end{abstract}
\section{Introduction}
Weil bundles, introduced by André Weil, provide a powerful generalization of jet spaces and a rigorous foundation for synthetic differential geometry \cite{Wei, Mo}. Unlike traditional approaches relying on limits, Weil bundles enable a direct and axiomatic treatment of infinitesimal quantities, leading to a coordinate-free synthetic calculus on manifolds. This framework offers new perspectives on classical differential geometry problems and open doors to novel applications. For example, the ability of Weil bundles to encode higher-order geometric information makes them  valuable in geometric mechanics for modeling complex systems and in machine learning for constructing more informative latent spaces. However, the metrizability of \(M^\mathbf{A}\) is non-trivial, as it requires reconciling the discrete nature of nilpotent elements with the continuous structure of the underlying manifold. Furthermore, understanding dynamics on \(M^\mathbf{A}\) allows us to generalize classical stability analysis by considering not only the stability of points but also the stability of infinitesimal neighborhoods. This paper investigates the metrizability and dynamical properties of Weil bundles. Our investigation yields several results that advance the theory of Weil bundles. 
\begin{itemize}
	\item \text{Metrization:} For a smooth, compact manifold \(M\) and a Weil algebra \(\mathbf{A}\), we prove that the manifold \(M^\mathbf{A}\) of \(\mathbf{A}\)-points admits a canonical, weighted, complete metric \(\mathfrak{d}_w\) that encodes both base-manifold geometry and infinitesimal deformations (Lemma \ref{P-1},  Theorem \ref{Comp-1}).
	\item \text{Path Lifting:} We provide a constructive method for lifting curves from the base manifold \(M\) to the Weil bundle \(M^\mathbf{A}\), demonstrating that this lifting process preserves fundamental topological invariants such as connectedness and simple connectedness (Lemma \ref{L-1}).
		\item \text{Dynamics:} We characterize fixed points of diffeomorphisms on \(M^\mathbf{A}\) (Lemma \ref{L-1-2}, Lemma \ref{L-1-3}).
	\item \text{Topological Equivalence:} Analyzing the  Leray spectral sequence, we establish a strong topological equivalence between \(M\) and \(M^\mathbf{A}\), proving that \(H^*(M^\mathbf{A}) \cong H^*(M)\) and \(\pi_\ast(M^\mathbf{A}) \cong \pi_\ast(M)\) (Theorems \ref{Top-1} and \ref{Top-2}). This highlights that Weil bundles, while encoding infinitesimal information, do not alter the essential global topology.  Theorem \ref{Dyn-2} which allows us to transfer a wide range of properties from  $ Diff^\infty(M)$ to $ Diff^\infty(M^\mathbf A)$. 

\end{itemize}

\section{Preliminaries}
\subsection{Weil Algebras and Infinitesimal Structures}
Synthetic differential geometry  provides an axiomatic framework for infinitesimals by enriching smooth manifolds with nilpotent elements. Central to this framework are \emph{Weil algebras}, which formalize higher-order infinitesimal neighborhoods. 

\begin{definition}\cite{Ok}\label{def-weil-algebra}
A Weil algebra \(\mathbf{A}\) is a finite-dimensional, commutative, associative, unital \(\mathbb{R}\)-algebra of the form \(\mathbf{A} = \mathbb{R} \oplus \mathcal{A}\), where \(\mathcal{A}\) is a maximal ideal satisfying \(\mathcal{A}^{k+1} = 0\) for some \(k \geq 1\).
\end{definition}

\begin{example}\label{ex-weil-algebra}
	\begin{itemize}
		\item The algebra \(\mathbf{A}_1 = \mathbb{R}[\epsilon]/(\epsilon^2)\) models first-order infinitesimals. Its elements are dual numbers \(a + b\epsilon\), where \(\epsilon^2 = 0\). This corresponds to the tangent space \(T_xM\) at a point \(x \in M\).
		\item  The algebra \(\mathbf{A}_k = \mathbb{R}[\epsilon]/(\epsilon^{k+1})\) encodes \(k\)-th order jets, with \(\epsilon^{k+1} = 0\). 
	\end{itemize}
\end{example}

Weil algebras act as algebraic proxies for infinitesimal neighborhoods, enabling coordinate-free calculus. Their maximal ideal \(\mathcal{A}\) (nilpotent) represents "infinitesimally small" quantities, while the quotient \(\mathbf{A}/\mathcal{A} \cong \mathbb{R}\) recovers the base point.

\subsection{Weil Bundles}

\subsubsection{Infinitely near points} Let $M$ be a smooth manifold of dimension $m$, and let $I(x) \subset C^\infty(M)$ denote the ideal of smooth functions on $M$ that vanish at a point $x \in M$. For each natural number $k$, we denote by $\mathcal{I}^k(x)$ the ideal of smooth functions vanishing at $x$, along with all their derivatives up to order $k$. These ideals are related to jet spaces, as demonstrated by the isomorphism
$$J^k_x(M, \mathbb R) := C^\infty(M)/\mathcal I^k(x) \approx  \mathbb R[[X_1, \dots, X_n]]/(X_1,\cdots, X_n)^{k+1} =: W_n^k,$$
where $J^k_x(M, \mathbb R)$ is the space of $k-$jets of functions from $M$ to $\mathbb R$ at $x$, and $W_n^k$ is a truncated polynomial algebra. This isomorphism implies that for any given formal power series $S$, there exists a smooth function $f$ whose Taylor expansion near $x$ is precisely $S$. 
\begin{definition}\label{def-2}
	Let $M$ be a smooth manifold and $\mathbf{A}$ a Weil algebra. An \textbf{infinitely near point to} $x \in M$ \textbf{of kind A} is a \textbf{smooth} morphism of $\mathbb{R}$-algebras $\phi : C^\infty(M) \rightarrow \mathbf{A} = \mathbb{R} \oplus \mathcal{A}$ such that the following diagram commutes:
	\[
	\begin{tikzcd}
		&   \mathbb R \oplus \mathcal{A} \arrow{dr}{pr_{\mathbb R}} \\
		C^\infty(M)  \arrow{ur}{\phi} \arrow{rr}{ev_x} && \mathbb R,
	\end{tikzcd}
	\]
	where $ev_x$ is the evaluation map, defined by $ev_x(f) := f(x)$, and $pr_{\mathbb R}$ is the projection onto the real part of $\mathbf A$.
\end{definition}
This definition captures the idea of studying functions on $M$ in an infinitesimal neighborhood of the point $x$. The morphism \(\phi\) maps functions on \(M\) to elements of the Weil algebra \(\mathbf{A}\), effectively probing the infinitesimal structure around \(x\). Let $M^\mathbf A _x$ denote the set of all infinitely near points to $ x\in M$ of kind $\mathbf A$. Each element $\zeta\in M^\mathbf A _x$ can be expressed as 
\begin{equation}\label{Decom}
\zeta(f) =  ev_x(f)+ L_\zeta(f),
\end{equation}
 where $L_\zeta:C^\infty(M)\rightarrow  \mathcal{A}$ is a linear map satisfying  the Leibniz rule:
\begin{equation}\label{Eq-0}
	L_\zeta(fg + \lambda h) = L_\zeta(f) g(x) + f(x)L_\zeta(g) + L_\zeta(f)L_\zeta(g) + \lambda L_\zeta(h),
\end{equation}
for all $f, g, h \in C^\infty(M)$ and $\lambda\in \mathbb R$ (see Morimoto \cite{Mo}). This Leibniz rule reflects the algebraic structure of the Weil algebra and the way infinitely near points interact with function multiplication.

\subsubsection{Weil bundles as manifolds}

Given a manifold \(M\) and a Weil algebra \(\mathbf{A}\), we can construct the Weil bundle \(M^\mathbf{A}\), which consists of all infinitely near points to \(M\) of kind \(\mathbf{A}\). There exists a natural projection map \(\pi_{\mathbf A }: M^\mathbf A \rightarrow M\), mapping each infinitely near point to its base point. The triple  \(( M^\mathbf A,\pi_{\mathbf A } , M)\) equipped with the bundle topology, is known as the bundle of \(\mathbf A\)-points near to points in \(M\). The following proposition provides a way to construct Weil bundles.

\begin{proposition}\label{prop-1}\cite{Mo}
	Let $M$ be a smooth manifold and $\mathbf A$ a Weil algebra. The collection $M^\mathbf A$ of all infinitely near points of $M$ is a smooth manifold. If $V$ is a vector space, then $V^\mathbf A\approx V\oplus_{\mathbb R} \mathcal A$. Furthermore, we have $(M\times N)^\mathbf A\approx M^\mathbf A\times N^\mathbf A$.
\end{proposition}
\subsection{Local coordinates on \(M^\mathbf{A}\)}

The goal of this section is to explicitly define local charts for \(M^\mathbf{A}\), as these are essential for this work. 
\subsubsection{\(M^\mathbf{A}\) as an \(\mathbf{A}\)-manifold}

Assume that \(\dim M = m\) and \(\dim \mathbf A = l\).  Then, the \(\mathbf{A}\)-dimension of \(M^\mathbf A\) is \(m\).  Let \(x \in M\), and let \((U, \phi := (\phi_1,\dotsm,\phi_m))\) be a local chart on \(M\) around \(x\).  Also, let \(\alpha_1,\dots,\alpha_l\) be a fixed basis of the Weil algebra \(\mathbf A\).  Then, we can define a map
$
\phi^\mathbf A:  \pi_{\mathbf A }^{-1}(U)\rightarrow  \mathbf A^m, \quad \varepsilon\mapsto (\varepsilon(\phi_1),\dots,\varepsilon(\phi_m) ).
$
This map \(\phi^\mathbf A\) is a diffeomorphism, which shows that \((\pi_{\mathbf A }^{-1}(U), \phi^\mathbf A )\) is a local chart on \(M^\mathbf A\) around \(\varepsilon\). Thus, \(M^\mathbf A\) is an \(\mathbf{A}\)-manifold of dimension \(m\).

\subsubsection{\(M^\mathbf{A}\) as a Real-manifold }

 Let \(\epsilon\in M^\mathbf A\), so \(\epsilon\in M^\mathbf A_x\) for some \(x\in M\). Let $x\in M$, and choose a chart \((U, \phi_1, \dots, \phi_m)\) on \(M\) around \(x\). We can express \(\epsilon(\phi_i)\) with respect to the basis \(\{\alpha_1, \dots,\alpha_{l}\}\) as
$
\epsilon(\phi_i) = \sum_{k = 1}^{l} x_{ik}(\epsilon)\alpha_k,
$
for each \(i = 1,\dots, m\), where the \(x_{ik}(\epsilon)\) are real numbers. Thus, we can construct a local chart on \(M^\mathbf A\) as
$
(\pi^{-1}(U), x_{1,1},\dots x_{m,1}, x_{1,2},\dots x_{m,2}, \dots,  x_{1,l},\dots x_{m,l}  ).
$ Without loss of generality, we can assume that \(\alpha_1 = 1_\mathbf A\) is the multiplicative identity in \(\mathbf A\). At first glance, it might appear that this chart depends on the choice of \(\epsilon\). However, we can show that this is not the case. For any \(\epsilon_u\) with \(u = 1, 2\), we have
\begin{equation}\label{E1}
	x_{ij}(\epsilon_1) = x_{ij}(\epsilon_2) + \alpha_j^\ast(t_1^i -t_2^i),
\end{equation}
where \(\epsilon_u(\phi_i) = \phi_i(x) + t_u^i\) for \(u =1, 2\), and  $\alpha_j^\ast:  \mathbf A \longrightarrow \mathbb R,$ is the dual element of $\alpha_j$.
On the other hand, we can write
$
\epsilon_u(\phi_i) = \phi_i(x) + t_u^i = x_{i1}(\epsilon_u) + \sum_{k = 2}^{\dim(A)} x_{ik}(\epsilon)\alpha_k,
$
which implies
\begin{equation}\label{E2}
	x_{i1}(\epsilon_u)  = \phi_i(x),
\end{equation}
for all \(i = 1,\dots, m\). Combining (\ref{E1}) and (\ref{E2}) yields
$
	t_1^i  = t_2^i,
$
for all \(i = 1,\dots, m\). Therefore, the chart
$
(\pi^{-1}(U), x_{1,1},\dots x_{m,1}, x_{1,2},\dots x_{m,2}, \dots,  x_{1,l},\dots x_{m,l}  )
$
depends only on \(x\), and not on the specific choice of \(\epsilon\in \pi^{-1}(U)\).

\subsubsection{Example: Coordinate transitions on \((\mathbb{R}^2)^{\mathbb{R}[\epsilon]/(\epsilon^3)}\)}
To illustrate coordinate transitions, consider the case where \(M = \mathbb{R}^2\) and \(\mathbf{A} = \mathbb{R}[\epsilon]/(\epsilon^3)\), where \(\epsilon^3 = 0\). This Weil algebra encodes second-order infinitesimal information.

\subsubsection*{\(\mathbf{A}\)-Coordinates}

Let \((U, \phi = (x, y))\) and \((V, \psi = (u, v))\) be two overlapping charts on \(\mathbb{R}^2\), where \(x, y\) are the standard coordinates and \(u = x^2\), \(v = y + x\). Then, on the overlap \(U \cap V\), we have the transition map:
$
\psi \circ \phi^{-1}(x, y) = (x^2, y + x).
$ Now, consider the Weil bundle \((\mathbb{R}^2)^{\mathbb{R}[\epsilon]/(\epsilon^3)}\), pick \(\xi \in (\mathbb{R}^2)^{\mathbb{R}[\epsilon]/(\epsilon^3)}\). With respect to the chart \((U, \phi)\), we have \(\mathbf{A}\)-valued coordinates:
$
(\xi(x), \xi(y)) = (x_0 + x_1\epsilon + x_2\epsilon^2, y_0 + y_1\epsilon + y_2\epsilon^2),
$ where \(x_0, x_1, x_2, y_0, y_1, y_2 \in \mathbb{R}\). Similarly, with respect to \((V, \psi)\), we have:\\
$
(\xi(u), \xi(v)) = (u_0 + u_1\epsilon + u_2\epsilon^2, v_0 + v_1\epsilon + v_2\epsilon^2).
$ The transition map in \(\mathbf{A}\)-valued coordinates is given by:
$
(\xi(u), \xi(v)) = (\xi(x^2), \xi(y + x)).
$ Expanding this, we get:

\[
\xi(u) = (x_0 + x_1\epsilon + x_2\epsilon^2)^2 = x_0^2 + 2x_0x_1\epsilon + (2x_0x_2 + x_1^2)\epsilon^2,
\]
\[
\xi(v) = (y_0 + y_1\epsilon + y_2\epsilon^2) + (x_0 + x_1\epsilon + x_2\epsilon^2) = (y_0 + x_0) + (y_1 + x_1)\epsilon + (y_2 + x_2)\epsilon^2.
\]

Thus, the coordinate transition is:

\[
(u_0, u_1, u_2, v_0, v_1, v_2) = (x_0^2, 2x_0x_1, 2x_0x_2 + x_1^2, y_0 + x_0, y_1 + x_1, y_2 + x_2).
\]

\subsubsection*{Real-Coordinates}
In real-valued coordinates, \(\xi\) is represented as \((x_0, x_1, x_2, y_0, y_1, y_2)\) in the chart \((U, \phi)\) and as\\ \((u_0, u_1, u_2, v_0, v_1, v_2)\) in the chart \((V, \psi)\). The transition between these real-valued coordinates is given by the same expressions we derived above:
\[
u_0 = x_0^2, \quad u_1 = 2x_0x_1, \quad u_2 = 2x_0x_2 + x_1^2,
\]
\[
v_0 = y_0 + x_0, \quad v_1 = y_1 + x_1, \quad v_2 = y_2 + x_2.
\]

This example illustrates how the coordinate transition in the Weil bundle is determined by the transition map on the base manifold, along with the algebraic structure of the Weil algebra. \\

The following proposition describes the relationship between the boundary of the base manifold \(M\) and the boundary of the Weil bundle \(M^\mathbf{A}\).

\begin{proposition}\label{P-0}
	Let \(M\) be a smooth manifold with boundary and \(\mathbf A\) a Weil algebra. Then, regarding \(M^\mathbf A\) as a real manifold, we have \(\partial( M^\mathbf A) \neq\emptyset\), and
	$
	\bigcup_{x\in \partial M} M^\mathbf A_x \subseteq \partial( M^\mathbf A).$	In other words, if you are in the fibers over points of \(\partial M\)  then you lie on the boundary too.

\end{proposition}

\begin{proof}
	Without loss of generality, we equip \(\mathbf A\) with a fixed basis \(1_{\mathbf A},\dots,\alpha_l\). Assume that \(\partial M \neq \emptyset\), and pick a point \(x \in \partial M\). Then, there exists a chart \((U, x_1,\dots, x_{\dim M})\) on \(M\) such that \(x_1(x) = 0\). Now, consider an infinitely near point \(\xi\in M^\mathbf A_x\). We can express \(\xi(x_i)\) as
	\[
	\xi(x_i) = \sum_{j = 1}^{\dim \mathbf A} x_{ij}(\xi)\alpha_j = x_i(x)1_{\mathbf A} + L_{\xi}(x_i),
	\]
	for each \(i = 1,\dots, \dim M\). Thus, \(x_{i1}(\xi) = x_i(x)\) for all \(i\), which implies that \(x_{11}(\xi) = x_1(x) = 0\). This shows that \(\partial M^{\mathbf A} \neq \emptyset\) (since the coordinate \(x_{11}\) vanishes), and that
	$
	\dim (\partial M^{\mathbf A}) = \dim M^{\mathbf A} - 1.
	$
	Furthermore, it follows that
	$
	\bigcup_{x\in \partial M} M^\mathbf A_x \subseteq \partial( M^\mathbf A).
	$
\end{proof}

\section{Topologies on \(M^\mathbf{A}\)}
To study the topological properties of \(M^\mathbf{A}\), we will introduce a metric that is compatible with its manifold structure. This will allow us to use tools from analysis to investigate properties such as completeness and connectedness. Assume that \(\alpha_1 = 1_\mathbf A, \alpha_2,\dots,\alpha_l\) is a basis of \(\mathbf A\) such that \(\alpha_2,\dots,\alpha_l\) forms a basis of the maximal ideal \(\mathcal A\). We equip the vector space \(\mathbf A\) with a norm \(\Vert \cdot \Vert_\mathbf A\).

\subsection{Weighted norms}

To define a suitable metric on \(M^\mathbf{A}\), we first introduce the concept of a weighted norm on the Weil algebra \(\mathbf{A}\). This allows us to control the relative importance of the infinitesimal directions in the Weil algebra. Let \(w = (w_1,\dots, w_l)\) be a vector of positive real numbers, representing the weights for the infinitesimal directions. This choice of weights provides a scaling when considering infinitesimals.   For an element \(a\in \mathbf A\), written as \(a = a_1\alpha_1 + \dots + a_l\alpha_l\), the weighted norm \(\Vert a \Vert_w\) is defined as:
\begin{equation}
	\Vert a \Vert_w = |a_1| + \left\Vert \sum_{i=2}^l a_iw_i\alpha_i \right\Vert_\mathbf A,
\end{equation}
where \(\Vert \cdot \Vert_\mathbf A\) is a norm on \(\mathbf A\). The weighted norm \(\Vert \cdot \Vert_w\) assigns a weight of $1$ to the real part of the Weil algebra and scales the infinitesimal part by the weights \(w\).  The weighted norm can be supported by the fact that in many geometric and physical applications, different infinitesimal directions may have different physical scales or significance. 

\subsection{Controlling the maps \(f \mapsto L_\xi(f)\)}

To establish the metrizability of \(M^\mathbf{A}\), it is crucial to show that the linear map \(f \mapsto L_\xi(f)\) is, in some sense, controlled by the function's size. Assume that \(M\) is compact.

\subsection*{Local coordinates}
Let $(U, \phi = (x^1, \dots, x^m))$ be a chart on $M$ centered at $x \in M$ (i.e., $\phi(x) = 0$). For $\xi \in M_x^\mathbf{A}$, the coordinates of $\xi$ are defined as:
$ \epsilon^i := \xi(x^i) \quad \text{for } i = 1, \dots, m. $
Since $x^i(x) = 0$, we have $L_\xi(x^i) = \xi(x^i) - x^i(x) = \epsilon^i \in \mathcal{A}$. 
For a smooth function $f \in C^\infty(M)$, we have:
\begin{equation}
	\label{eq:L_xi_f}
	L_\xi(f) = \sum_{1 \leq |\mathbf{i}| \leq k} \frac{1}{|\mathbf{i}|!} \partial^{\mathbf{i}} f(x) \epsilon^{\mathbf{i}},
\end{equation}
where: $\mathbf{i} = (i_1, \dots, i_r)$ is a multi-index with $|\mathbf{i}| = r$,  $\partial^{\mathbf{i}} f(x) = \frac{\partial^r f}{\partial x^{i_1} \dots \partial x^{i_r}}(x)$, and 
	 $\epsilon^{\mathbf{i}} = \epsilon^{i_1} \dots \epsilon^{i_r}$. Now, we consider two function spaces:
\begin{enumerate}
	\item The unit ball with uniform bound on derivatives:
	$$ \mathfrak{B}^1_k(M) = \left\{ f \in C^\infty(M) : \max_{|\mathbf{i}| \leq k} \sup_{x \in M} |\partial^{\mathbf{i}} f(x)| \leq 1 \right\},$$
		where \(k\) matches the order of \(\mathbf{A}\). 
	\item The standard unit ball with the uniform supremum norm:
	$$ \mathfrak{B}^1_0(M) = \{ f \in C^\infty(M, \mathbb{R}) : |f|_0 \leq 1 \}, $$
	where $|f|_0 := \sup_{x \in M} |f(x)|$.
\end{enumerate}
Let $\vartheta, \varepsilon \in M^\mathbf{A}$. We define two pseudo-metrics on $M^\mathbf{A}$:

\begin{enumerate}
	\item \textbf{Using the $C^k$-unit ball:}
	\begin{equation}\label{M-1}
		\mathfrak d_w(\vartheta, \varepsilon) = d_g(\pi_{\mathbf A }(\vartheta), \pi_{\mathbf A }( \varepsilon)) + \sup_{f\in \mathfrak {B}^1_k(M)}\Vert L_\vartheta(f) - L_\varepsilon(f)\Vert_w,
	\end{equation}
	\item \textbf{Using factorial decay weights:} Suppose  $w_{\mathbf{i}} \sim \frac{1}{|\mathbf{i}|!}$ , let us define factorial decay in the weights $\overline{w} = (w_1,\dots, w_l)$ such that $w_{|\mathbf{i}|} = 1/|\mathbf{i}|!$.
	\begin{equation}\label{M-2}
		\mathfrak d_{\overline{w}}(\vartheta, \varepsilon) = d_g(\pi_{\mathbf A }(\vartheta), \pi_{\mathbf A }( \varepsilon)) + \sup_{f\in \mathfrak {B}^1_0(M)}\Vert L_\vartheta(f) - L_\varepsilon(f)\Vert_{\overline{w}},
	\end{equation}
	
	where $d_g$ is the Riemannian distance on $M$.
\end{enumerate}

\begin{lemma}\label{P-1}
	Let \(M\) be a smooth compact manifold equipped with a Riemannian metric \(g\) and \(\mathbf A\) a Weil algebra. With the above notation, consider \(\mathfrak{d}_w\) and  \(\mathfrak{d}_{\overline w}\) defined in (\ref{M-1}) and (\ref{M-2}),
	then :
	\begin{enumerate}
		\item The function $\mathfrak d_w$ and   \(\mathfrak{d}_{\overline w}\) define metrics on $M^\mathbf A$.
		\item The topology induced by either function are equivalent.
		\item The topology induced by  each metric is equivalent to the standard manifold topology on $M^\mathbf A$.
	\end{enumerate}
\end{lemma}

\begin{proof}

\textbf{ Verify that $\mathfrak{d}_w$ and $\mathfrak{d}_{\overline{w}}$ are Metrics.} We must verify the properties of a metric: non-negativity, symmetry, triangle inequality, and identity of indiscernibles.\\
\textbf{Non-negativity and Symmetry.}
Both $\mathfrak{d}_w$ and $\mathfrak{d}_{\overline{w}}$ are non-negative because both $d_g$ and the supremum terms are non-negative by construction.  Symmetry holds because $d_g(\vartheta, \varepsilon) = d_g(\varepsilon, \vartheta)$ and
$$ \sup_{f} \|L_\vartheta(f) - L_\varepsilon(f)\|_w = \sup_{f} \|L_\varepsilon(f) - L_\vartheta(f)\|_w, $$
since taking the absolute value inside the norm makes the differences symmetric.\\
\textbf{Triangle Inequality.} For $\vartheta, \varepsilon, \zeta \in M^\mathbf{A}$, we need to show that
$ \mathfrak{d}_w(\vartheta, \zeta) \leq \mathfrak{d}_w(\vartheta, \varepsilon) + \mathfrak{d}_w(\varepsilon, \zeta). $ We have:
$$ \mathfrak{d}_w(\vartheta, \zeta) = d_g(\pi(\vartheta), \pi(\zeta)) + \sup_{f \in \mathfrak{B}^1_k(M)} \|L_\vartheta(f) - L_\zeta(f)\|_w. $$

By the triangle inequality for $d_g$, we can split the Riemannian distance term:
$$ d_g(\pi(\vartheta), \pi(\zeta)) \leq d_g(\pi(\vartheta), \pi(\varepsilon)) + d_g(\pi(\varepsilon), \pi(\zeta)). $$

Using the linearity of $L_\xi$, we have $L_\vartheta(f) - L_\zeta(f) = (L_\vartheta(f) - L_\varepsilon(f)) + (L_\varepsilon(f) - L_\zeta(f))$. Thus,
\begin{align*}
	\sup_{f \in \mathfrak{B}^1_k(M)} \|L_\vartheta(f) - L_\zeta(f)\|_w &\leq \sup_{f \in \mathfrak{B}^1_k(M)} \|L_\vartheta(f) - L_\varepsilon(f)\|_w + \sup_{f \in \mathfrak{B}^1_k(M)} \|L_\varepsilon(f) - L_\zeta(f)\|_w.
\end{align*}
Combining these inequalities, we obtain
$ \mathfrak{d}_w(\vartheta, \zeta) \leq \mathfrak{d}_w(\vartheta, \varepsilon) + \mathfrak{d}_w(\varepsilon, \zeta). $
The same argument applies to $\mathfrak{d}_{\overline{w}}$.
\subsubsection*{Identity of indiscernibles:}
If $\mathfrak{d}_w(\vartheta, \varepsilon) = 0$, then $d_g(\pi(\vartheta), \pi(\varepsilon)) = 0 \implies \pi(\vartheta) = \pi(\varepsilon) = x$ for some $x \in M$, and\\ 
	 $\sup_{f \in \mathfrak{B}^1_k(M)} \|L_\vartheta(f) - L_\varepsilon(f)\|_w = 0 \implies \|L_\vartheta(f) - L_\varepsilon(f)\|_w = 0$ for all $f \in \mathfrak{B}^1_k(M)$. 
This implies $L_\vartheta(f) = L_\varepsilon(f)$ for all $f \in \mathfrak{B}^1_k(M)$. 
Since $\vartheta = ev_x + L_\vartheta$ and $\varepsilon = ev_x + L_\varepsilon$, the fact that $L_\vartheta(f) = L_\varepsilon(f)$ for all $f \in \mathfrak{B}^1_k(M)$ implies $\vartheta = \varepsilon$. A similar argument holds for $\mathfrak{d}_{\overline{w}}$. Therefore, $\mathfrak{d}_w$ and $\mathfrak{d}_{\overline{w}}$ satisfy all the properties of a metric.\\
\textbf{Show topological equivalence of $\mathfrak{d}_w$ and $\mathfrak{d}_{\overline{w}}$:} The metrics $\mathfrak{d}_w$ and $\mathfrak{d}_{\overline{w}}$ differ only in the weighted norms $\|\cdot\|_w$ and $\|\cdot\|_{\overline{w}}$. Since $\mathbf{A}$ is a finite-dimensional vector space over $\mathbb{R}$ (because $\mathcal{A}^{k+1} = 0$), all norms on $\mathbf{A}$ are equivalent.  Therefore, there exist constants $C_1, C_2 > 0$ such that:
$$ C_1 \|a\|_w \leq \|a\|_{\overline{w}} \leq C_2 \|a\|_w \quad \forall a \in \mathbf{A}. $$ 
This norm equivalence ensures that $\mathfrak{d}_w$ and $\mathfrak{d}_{\overline{w}}$ generate the same topology.  Specifically, for any $\vartheta, \varepsilon \in M^{\mathbf{A}}$, we have:
$C_1 \mathfrak{d}_w(\vartheta, \varepsilon) \leq \mathfrak{d}_{\overline{w}}(\vartheta, \varepsilon) \leq C_2 \mathfrak{d}_w(\vartheta, \varepsilon).$
This implies that if a sequence of points converges to $\vartheta$ with respect to $\mathfrak{d}_w$, it also converges to $\vartheta$ with respect to $\mathfrak{d}_{\overline{w}}$, and vice versa.\\
\textbf{Align the metric topology with the standard manifold topology:}
 Now, we  need to show that a sequence \(\{\xi_n\}\) converges to \(\xi\) with respect to the metric \(\mathfrak d_w\) if and only if it converges to \(\xi\) with respect to the original manifold topology on \(M^\mathbf A\).  Suppose \(\{\xi_n\}\) is a sequence of points in \(M^\mathbf A\) such that \(\xi_n \to \xi\) as \(n\to \infty\) with respect to the original topology. This means that for any open neighborhood \(U_\xi\) of \(\xi\) in \(M^\mathbf A\), there exists an integer \(n_0\) such that \(\xi_n \in U_\xi\) for all \(n \geq n_0\). Consider the decompositions
		$
		\xi_n = ev_{\pi_{\mathbf A }(\xi_n)} + L_{\xi_n},$ and $ \xi = ev_{\pi_{\mathbf A }(\xi)} + L_\xi.
		$ Since \(\pi_{\mathbf A }\) is continuous, the sequence \(x_n := \pi_{\mathbf A }(\xi_n)\) converges to \(\pi_{\mathbf A }(\xi) =: x \in M\) with respect to the metric \(d_g\). Since \(M\) is compact, there exists a local chart \((U, (x_1,\dots, x_{\dim M}))\) around \(x\) such that \((\pi_{\mathbf A }^{-1}(U), (x_{ij}(\xi))_{1\leq j\leq\dim \mathbf A, 1\leq i\leq \dim M})\) forms a local chart around \(\xi\). Here, \(x_{ij}(\xi)\) represents the \(j\)-th component of \(L_\xi(x_i)\) in some basis of the Weil algebra \(\mathbf{A}\). The manifold topology implies that \(x_{ij}(\xi_n) \to x_{ij}(\xi)\) as \(n \to \infty\) for each \((i,j)\), where the convergence occurs in \(\mathbb{R}\). This convergence \(x_{ij}(\xi_n) \to x_{ij}(\xi)\) is precisely the convergence \(L_{\xi_n}(x_i) \to L_\xi(x_i)\) in \(\mathbf{A}\). Equation (\ref{Eq-0}) can be used to deduce that for each \(i\) and each natural number \(\lambda\), \(L_{\xi_n}(x_i^\lambda) \to L_{\xi}(x_i^\lambda)\) as \(n \to \infty\). Now, consider the Taylor expansion of a smooth function \(f \in \mathfrak {B}^1_k(M)\) near \(\pi_{\mathbf A }(\xi) = x\):
		\[ f(y) = f(x) + \sum_{1 \le |\mathbf{i}| \le k} \frac{1}{|\mathbf{i}|!} \partial^{\mathbf{i}} f(x) (y - x)^{\mathbf{i}} + R_{k+1}(y),
		\]
		where the remainder \(R_{k+1}\) vanishes when acted on by the linear map and  \(k\) matches the order of \(\mathbf{A}\). With this information, the norm \( \Vert  L_{\xi_n}(f) - L_\xi(f) \Vert_w \) can be estimated as
		
		\[
		\Vert L_{\xi_k}(f) -  L_\xi(f) \Vert_w =  \Vert \sum_{1 \le |\mathbf{i}| \le k} \frac{1}{|\mathbf{i}|!} \partial^{\mathbf{i}} f(x_n) L_{\xi_n}(x^{\mathbf{i}})  -  \sum_{1 \le |\mathbf{i}| \le k} \frac{1}{|\mathbf{i}|!} \partial^{\mathbf{i}} f(x) L_{\xi}(x^{\mathbf{i}}) \Vert_w.
		\]
		Because the chart is smooth, the function is smooth and \( x_n \rightarrow x \) we know that \( \partial^{\mathbf{i}} f(x_k)  \rightarrow \partial^{\mathbf{i}} f(x)\).
		Therefore \( \mathfrak d_w(\xi_n , \xi) \rightarrow 0\) for \( n \rightarrow \infty\).  Conversely, assume that \(\mathfrak d_w(\xi_n, \xi) \to 0\) as \(n \to \infty\). This implies that \(d_g(\pi_{\mathbf A }(\xi_n), \pi_{\mathbf A }(\xi)) \to 0\), so \(\pi_{\mathbf A }(\xi_n) \to \pi_{\mathbf A }(\xi)\) in \(M\). Also, it implies that \(\sup_{f\in \mathfrak {B}^1_k(M)}\Vert L_{\xi_n}(f) -   L_\xi(f)\Vert_w \to 0\).
		Since convergence in the manifold topology is defined chartwise, it is enough to consider coordinate functions:
		$ \|L_{\xi_k}(x_i)-L_{\xi}(x_i)\|_w\leq \sup_{f\in \mathfrak {B}^1_k(M)}\Vert L_{\xi_n}(f) -   L_\xi(f)\Vert_w\rightarrow 0$.
		But \(L_{\xi_n}(x_i)\) and \(L_{\xi}(x_i)\) are precisely the components of \(\xi_n\) and \(\xi\) in the local chart around \(\xi\). Hence, \(\xi_n \to \xi\) in the manifold topology. We have shown that \(\mathfrak d_w\) is a metric on \(M^\mathbf A\) and that the metric topology induced by \(\mathfrak d_w\) is equivalent to the manifold topology on \(M^\mathbf A\). Therefore, \(M^\mathbf A\) is metrizable.

\end{proof}
\begin{theorem}\label{Comp-1}
	If \(M\) is a compact, connected, smooth manifold equipped with a Riemannian metric, and \(\mathbf{A}\) is a Weil algebra, then the metric space \((M^\mathbf{A}, \mathfrak d_w)\) is complete.
\end{theorem}

\begin{proof}
	Let \(\{\xi_k\}\) be a Cauchy sequence in \((M^\mathbf{A}, \mathfrak d_w)\). This means that for any \(\epsilon > 0\), there exists a natural number \(N\) such that for all \(m, k > N\), we have:
	$
	\mathfrak d_w(\xi_m, \xi_k) < \epsilon.
	$
	Our goal is to show that \(\{\xi_k\}\) converges to a limit in \(M^\mathbf{A}\). 
	The projection map \(\pi_{\mathbf A } :(M^\mathbf{A}, \mathfrak d_w)\rightarrow (M, d_g)\) is $1-$Lipschitz, meaning that
	$
	d_g(\pi_{\mathbf A }(\xi), \pi_{\mathbf A }(\eta)) \leq \mathfrak d_w(\xi, \eta),
	$
	for all \(\xi, \eta \in M^\mathbf A\).  Therefore, for any \(\epsilon > 0\),  if \(m, k > N\), we have:
	\[
	d_g(\pi_{\mathbf A }(\xi_m), \pi_{\mathbf A }(\xi_k)) \leq \mathfrak d_w(\xi_m, \xi_k) < \epsilon.
	\]
	This shows that the sequence \(y_k := \pi_{\mathbf A }(\xi_k)\) is a Cauchy sequence in \((M, d_g)\). 
	Since \(M\) is compact, it is complete.  Therefore, the Cauchy sequence \(\{y_k\}\) converges to some limit \(\bar y\in M\).  
	Now, we need to show that the infinitely near parts of the \(\xi_k\) also converge. To do this, we consider the map
	\[
	\mathcal M :(M^\mathbf{A}\times C^\infty(M), d^\ast) \rightarrow (\mathbf{A}, \Vert .\Vert_w) , \qquad (\xi, f)\mapsto L_\xi(f),
	\]
	where \(d^\ast := \mathfrak d_w + |.|_0\) is a metric on the Cartesian product \(M^\mathbf{A}\times C^\infty(M) \) and \(|.|_0\) is the uniform supremum metric. 
	To prepare our analysis, show that  the map \(\mathcal M\) is 1-Lipschitz. Alternatively, consider some elements $\xi, \zeta \in M^\mathbf{A}$ and $ f,g \in C^\infty(M)$. Then compute,
	$$
	\Vert \mathcal M(f, \xi) - \mathcal M(g,\zeta)\Vert_w := \Vert L_\xi(f) - L_\zeta(g)  \Vert_w
	 \leqslant   d^\ast ((f, \xi) , (g, \zeta) ).
	$$	
	
	For each \(f\in C^\infty(M)\), the sequence \(\theta_k := \mathcal M(\xi_k, f) = L_{\xi_k}(f) \) is a Cauchy sequence in \(\mathbf A\). (Because \(\{\xi_k\}\) is Cauchy in \(M^\mathbf{A}\) and \(\mathcal{M}\) is Lipschitz.) The Weil algebra \(\mathbf A \) is finite dimensional, it is complete guaranteeing that \(\theta_k = \mathcal M(\xi_k, f) \) converges to some \(\mu(f)\in \mathbf A \) for each \(f\). 
	By construction, the limit \(\mu(f)\) satisfies the adapted Leibniz rule
	$
		\mu(fg + \lambda h) = \mu(f) g(\bar y) + f(\bar y)\mu(g) + \nu(f)\mu(g) + \lambda \mu(h),
	$
	for all \(f , g, h\in C^\infty(M)\), and \(\lambda \in \mathbb R\). Therefore, the map \(\bar \xi(f) := f(\bar y) + \mu(f)\) belongs to \(M^\mathbf{A}_{\bar y}\).  Finally, we show that \(\mathfrak d_w(\xi_k,\bar\xi)\rightarrow 0\) as \(k\rightarrow\infty\). This completes the proof that \((M^\mathbf{A}, \mathfrak d_w)\) is complete.
\end{proof}
\begin{table}[h!]
	\centering
	\begin{tabular}{@{} p{2cm} p{5cm} p{8cm} @{}} 
		\toprule
		Feature & Carnot-Carathéodory Metrics & Metric \(\mathfrak{d}_w\) \\
		\midrule
		Primary Focus & Motion Constraints & Infinitesimal Structure \\
		Construction & Path Optimization & Direct Formula \\
		Geometry & Singular & Riemannian-like (Anisotropic) \\
		Completeness & Generally Incomplete (unless compact or strong bracket-generating assumptions) & Complete (when $M$ is compact, Riemannian) \\
		Advantages & Enforcing Constraints & Analytical Tractability, Infinitesimal Encoding \\
		Limitations & Computationally Complex & Relies on Base Metric, Limited Constraints \\
		\bottomrule
	\end{tabular}
	\caption{Comparison of Carnot-Carathéodory Metrics and \(\mathfrak{d}_w\)}
	\label{tab:comparison}
\end{table}
\begin{remark}
	The choice of factorial decay effectively dampens the contribution of higher-order derivatives, preventing the supremum from diverging. Without the weighted norm, the  supremum term
	$
	\sup_{f \in \mathfrak{B}_0^1(M)} \|L_\vartheta(f) - L_\varepsilon(f)\|_w
	$
 can diverges due to unbounded derivatives. Functions in \(\mathfrak{B}_0^1(M)\)  can have arbitrarily large derivatives (e.g., \(f_n(x) = \sin(n^2 x)\) on \(M = S^1\), with \(|f_n|_0 \leq 1\) but \(\partial_x f_n \sim n^2\)).
 \end{remark}
\subsection{Lifting of curves}

In this section, we investigate how curves in the base manifold \(M\) can be lifted to curves in the Weil bundle \(M^\mathbf{A}\). This will allow us to relate the topological properties of \(M\) to those of \(M^\mathbf{A}\). Let \(\varepsilon_1, \varepsilon_2\in M^\mathbf A\).  We consider two cases:

 \textbf{Case 1: \(\varepsilon_1\) and \(\varepsilon_2\) lie in the same fiber.} Suppose that \(\varepsilon_1, \varepsilon_2\in M^\mathbf A_x\) for some \(x\in M\), with \(\varepsilon_1(f) = f(x) + L_1(f)\) and \(\varepsilon_2(f) = f(x) + L_2(f)\). Then, we can define a path \(\varepsilon_t\) in \(M^\mathbf A_x\) by
	\[
	\varepsilon_t(f) := f(x) + ((1-t) L_1(f) + t L_2(f)),
	\]
	for all \(t \in [0, 1]\) and all \(f\in C^\infty(M, \mathbb R)\).  It is easy to verify that \(\varepsilon_t \in M^\mathbf A_x\) for all \(t\).\\
	 \textbf{Case 2: \(\varepsilon_1\) and \(\varepsilon_2\) lie in different fibers.} Suppose that \(\varepsilon_1\in M^\mathbf A_x\) and \(\varepsilon_2\in M^\mathbf A_y\) with \(\varepsilon_1(f) = f(x) + L(f)\) and \(\varepsilon_2(f) = f(y) + G(f)\) for some \(x,y\in M\). Let \(\gamma\) be a smooth curve in \(M\) from \(x\) to \(y\). We can define a path \(\varepsilon_\gamma(t)\in M^\mathbf A_{\gamma(t)}\) by: 
	$
	\varepsilon_\gamma(t)(f) = f(\gamma(t)) + ((1-t)L(f) + t G(f)),
	$
	for all \(f\in C^\infty(M, \mathbb R)\) and for each \(t \in [0, 1]\). Let \(\{t_j\}\subseteq [0,1]\) be a sequence that converges to \(s\). Then, 
\begin{equation}
	\mathfrak  d_w(\varepsilon_\gamma(t_j), \varepsilon_\gamma(s)) := 	 d_g(\gamma(t_j),  \gamma(s)) + \sup_{f\in \mathfrak {B}^1_k(M)}\Vert  ((1-t_j)L(f) + t_j G(f)) - ((1-s)L(f) - s G(f))\Vert_w, 
\end{equation}
$$ =  d_g(\gamma(t_j),  \gamma(s)) + \sup_{f\in \mathfrak{B}^1_k(M)}\Vert  (s-t_j)L(f) + (-s + t_j)G(f)\Vert_w $$
$$ =  d_g(\gamma(t_j),  \gamma(s)) +  \sup_{h\in \mathfrak{B}^{1,j}_k(M)}\Vert  (L- G)(h)\Vert_w,$$
for $j$ sufficiently large and fixed, where $\mathfrak {B}^{1,j}_k(M)$ is the image of  $\mathfrak{B}^{1}_k(M)$ under the map $f\mapsto (s-t_j)f$. As $j$ tends to infinity,  $\mathfrak{B}^{1,j}_k(M)\longrightarrow \lbrace 0\rbrace$. Thus,
$$\lim_{j\longrightarrow\infty}\mathfrak  d_w(\varepsilon_\gamma(t_j), \varepsilon_\gamma(s)) := 	\lim_{j\longrightarrow\infty} d_g(\gamma(t_j),  \gamma(s)) + \lim_{j\longrightarrow\infty}\sup_{h\in \mathfrak {B}^{1,j}_k(M)}\Vert  (L- G)(h)\Vert_{w}
 = 0 + \sup_{h\in \lbrace 0 \rbrace}\Vert  (L- G)(h)\Vert_{w} = 0.$$

\begin{lemma}\label{L-1}
	Let \(M\) be a smooth compact manifold and \(\mathbf A\) a Weil algebra.
	\begin{enumerate}
		\item If \(M\) is connected, then so is \(M^\mathbf A\).
		\item If \(M\) is path-connected, then so is \(M^\mathbf A\).
		\item If \(M\) is simply connected, then so is \(M^\mathbf A\).
	\end{enumerate}
\end{lemma}

\begin{proof}
	\begin{enumerate}
		\item \textbf{Connectedness:}  
		
		Assume  $\mathcal{O}_1 $ and $\mathcal{O}_2 $  are two disjoint nonempty open subsets such that  $M^\mathbf A = \mathcal{O}_1\cup \mathcal{O}_2$. Since $M$
		is connected, fix a continuous section $S$ of  $\pi_\mathbf A$, and consider $\mathcal{U}_i := S^{-1}(\mathcal{O}_i)\cap M$, for $i = 1,2$. Each $\mathcal{U}_i$ is nonempty, and  open in $M$:   $ \mathcal{U}_1\cup \mathcal{U}_2 = S^{-1}(\mathcal{O}_1\cup \mathcal{O}_2)\cap M = S^{-1}(M^\mathbf A)\cap M = M$, and $ \mathcal{U}_1\cap \mathcal{U}_2 = S^{-1}(\mathcal{O}_1\cap \mathcal{O}_2)
		\cap M = S^{-1}(\emptyset)\cap M  = \emptyset.$ This contradicts the connectedness of $M$.
		\item \textbf{Path-Connectedness:} Let \(x, y\in M\), pick \(\xi\in M^\mathbf A_x\) and \(\eta\in M^\mathbf A_y\) such that \(\xi(f) = f(x) + L_\xi(f)\) and \(\eta(f) = f(y) + L_\eta(f)\) for all smooth functions \(f\) on \(M\). Consider \(\gamma\) to be a continuous curve in \(M\) from \(x\) to \(y\), and define
		$
		\xi_{\gamma}(s)(f) = f(\gamma(s)) + (1-s)L_\xi(f) + s L_\eta(f) \in M^{\mathbf A}_{\gamma(s)},
		$
		for all smooth functions \(f\) on \(M\). For each \(s \in [0,1]\), we have \(\xi_{\gamma}(0) = \xi\) and \(\xi_{\gamma}(1) = \eta\). The map \(s \mapsto \xi_{\gamma}(s)\) is a continuous path in \(M^\mathbf{A}\) with respect to the metric \(\mathfrak d_w\). This lifting map provides us with a continuous path between any two points in \(M^\mathbf{A}\).
		
		\item \textbf{Simple connectedness:} Consider two curves \(\xi_\gamma\) and \(\xi_\theta\) in \(M^\mathbf A\) with the same endpoints. Then, \(\gamma(t) := \pi_\mathbf A(\xi_\gamma(t))\) and \(\theta(t) := \pi_\mathbf A(\xi_\theta(t))\) are two curves in \(M\) with the same endpoints. Since \(M\) is simply connected, there exists a homotopy \(H(s,t)\) between \(\gamma\) and \(\theta\). We lift the homotopy $H$ into a homotopy $ \tilde H$  between \(\xi_\gamma\) and \(\xi_\theta\) defined as:\\
		$
		\tilde H(s, t)(f) = f(H(s,t)) + (1-s)L_{\xi_\gamma(t)}(f) + s L_{\xi_\theta(t)}(f),
		$
		for all smooth functions \(f\) on \(M\). It can be shown that the map \(\tilde H\) is continuous.
	\end{enumerate}
\end{proof}

\begin{figure}
	\begin{center}
		\begin{tikzpicture}[scale= 0.7]
			\draw[thick] (0,0) ellipse (2 and 1);
			\node at (1, 1) {\(M\)};
			
			\draw[thick, blue] (-1.5,0) to[out=30, in=150] (1.5,0);
			\node[blue] at (-0.4, 0.5) {\(\gamma(t)\)};
			
			\draw[thick, dashed] (0,0) -- (0,3);
			\draw[thick] (0,3) ellipse (0.5 and 0.2);
			\node at (0.7, 3.3) {\(M^\mathbf{A}_x\)};
			
			\draw[thick, red] (-1.5,3) to[out=30, in=150] (1.5,3);
			\node[red] at (0, 3.6) {\(\varepsilon_\gamma(t)\)};
			
			\draw[->, thick] (0,3) to[out=-90, in=90] (0,0);
			\node at (-0.5, 1.5) {\(\pi_\mathbf{A}\)};
			
			\draw[thick] (0,3) ellipse (2 and 1);
			\node at (0, 4.5) {\(M^\mathbf{A}\)};
		\end{tikzpicture}
	\end{center}
	\caption{Lifting of curve from  $M$  to  \(M^\mathbf{A}\).}
	\label{fig: Lifting-M- M^A}
\end{figure}
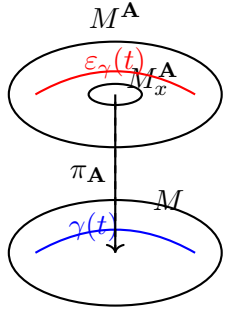
\begin{table}[h!]
	\centering
	\begin{tabularx}{\textwidth }{@{} p{4cm} p{4cm} p{4.9cm} p{3cm}@{}} 
		\toprule
		Property of M & Property of M\textsuperscript{A} & Conditions/Caveats & Result Proven? \\
		\midrule
		Smooth Manifold & Smooth Manifold & \text{By definition of Weil bundles} & Yes \\
		Complete & Complete (w.r.t. $\mathfrak{d}_w$) &  \text{M compact, connected, smooth}  & Yes \\
		Connected & Connected & & Yes \\
		Path-Connected & Path-Connected & & Yes \\
		Simply Connected & Simply Connected & & Yes \\
		Boundary $\neq \emptyset$ &  \text{Boundary $\neq \emptyset$}, $\bigcup_{x\in \partial M} M^\mathbf{A}_x \subseteq \partial( M^\mathbf{A})$&	 & Yes \\
		Riemannian Metric  & Riemannian Metric  & M and $M\textsuperscript{A}$   & Yes  \\
		\bottomrule
	\end{tabularx}
	\caption{Properties of Manifold M vs. Weil Bundle $M^\mathbf{A}$}
	\label{tab:M_vs_MA}
\end{table}

\section{Dynamics}

In this section, we investigate the relationship between the fixed points of a diffeomorphism of \(M\) and the fixed points of its Weil lifting on  \(M^\mathbf{A}\). This analysis will provide insights into how the infinitesimal structure captured by the Weil bundle affects the dynamics of diffeomorphisms. Let \(M, N\) be two smooth manifolds, and let \(\phi : M\longrightarrow N\) be a smooth map. Then, we can define a smooth map \(\phi^{\mathbf{A}} : M^\mathbf{A}\longrightarrow N^\mathbf{A}\) by : 
$
\phi^{\mathbf{A}}(\xi)(h) := \xi(h\circ \phi),
$
for all \(h\in C^\infty(N)\). This map \(\phi^{\mathbf{A}}\) is called the Weil lifting of \(\phi\).	We begin by showing that isometries are preserved under Weil lifting. 

\begin{lemma}\label{L-1-2}
	Let \((M, g_M), (N, g_N)\) be two smooth compact connected Riemannian manifolds, let \(\mathbf A\) be a Weil algebra, and let \(\phi : (M, g_M)\longrightarrow (N, g_N)\) be an isometry. Then,
	$
	\mathfrak d_w^N (\phi^{\mathbf{A}}(\xi_1), \phi^{\mathbf{A}}(\xi_2)) = \mathfrak  d_w^M(\xi_1, \xi_2),
	$
	for all \(\xi_1,\xi_2\in M^\mathbf{A}\).
\end{lemma}

\begin{proof}
	It suffices to show that
	\[
	\sup_{f\in \mathfrak {B}^1_K(M)}\Vert L_{\xi_1}(f) - L_{\xi_2}(f)\Vert_w = \sup_{h\in \mathfrak {B}^1_k(N)}\Vert L_{\phi^{\mathbf{A}}(\xi_1)}(h) - L_{\phi^{\mathbf{A}}(\xi_2)}(h)\Vert_w.
	\]
	
	Since \(\phi\) is an isometry, the pullback map \(\phi^\ast\) induces a bijection such that \(\phi^\ast(\mathfrak {B}^1_0(N)) = \mathfrak {B}^1_k(M)\). Thus,
	$$
		\sup_{f\in \mathfrak {B}^1_k(M)}\Vert L_{\xi_1}(f) - L_{\xi_2}(f)\Vert_w = \sup_{f\in \phi^\ast(\mathfrak {B}^1_k(N))}\Vert L_{\xi_1}(f) - L_{\xi_2}(f)\Vert_w \\
		= \sup_{h\in \mathfrak {B}^1_k(N)}\Vert L_{\xi_1}(h\circ\phi) - L_{\xi_2}(h\circ \phi)\Vert_w.
	$$
	
	To conclude, we need to show that \(L_{\xi_1}(h\circ\phi) = L_{\phi^{\mathbf{A}}(\xi_1)}(h)\). By the decomposition property,
	$
	\phi^{\mathbf{A}}(\xi_1)(h) = ev_{\pi_N(\phi^{\mathbf{A}}(\xi_i))}(h) + L_{\phi^{\mathbf{A}}(\xi_i)}(h).
	$
	On the other hand, by the definition of the map \(\phi^{\mathbf{A}}\), we have
	$
	\phi^{\mathbf{A}}(\xi_i)(h) = \xi_i(h\circ\phi) = ev_{\pi_M(\xi_i)}(h\circ\phi) + L_{\xi_i}(h\circ \phi).
	$
	By the uniqueness of the decomposition, we have \(L_{\phi^{\mathbf{A}}(\xi_i)}(h) = L_{\xi_i}(h\circ \phi)\) for each \(i = 1,2\). 
\end{proof}

We now establish a connection between the fixed points of a diffeomorphism on $M$ and the fixed points of its Weil lifting on \(M^\mathbf{A}\). We want to show that if a diffeomorphism \(\phi\) fixes a whole region \(W\) in \(M\), then its Weil lifting \(\phi^\mathbf{A}\) fixes all the infinitely near points "sitting above" that region in \(M^\mathbf{A}\). Conversely, if \(\phi^\mathbf{A}\) fixes some infinitely near point, its "shadow" on \(M\) must be a fixed point of \(\phi\). To do this, we will invoke the \textbf{Morimoto's Vanishing Lemma}. This is a fundamental result, due to Morimoto (see Remark 1.3 in \cite{Mo}), which states that:  If a smooth function \(f\) is identically zero in a neighborhood of a point \(x\in M\), then for all \(\xi\in M^\mathbf A_x\), we have \(\xi(f) = 0\).
\begin{lemma}\label{L-1-3}
	Let \(M\) be a smooth compact connected manifold, \(\mathbf A\) a Weil algebra, and \(\phi : M \longrightarrow M\) be a diffeomorphism. Let $ Fix(\phi) := \{x\in M \mid \phi(x) = x\}.$  If \(Fix(\phi)\) is nonempty and contains a neighborhood \(W\), then  \( Fix(\phi^{\mathbf{A}}) := \{\xi \in M^\mathbf{A} \mid \phi^\mathbf{A}(\xi) = \xi\}\), is nonempty and contains \(\bigcup_{x\in W} M^\mathbf{A}_x\). Conversely, if \(\phi : M \longrightarrow M\) is a diffeomorphism such that \(Fix(\phi^{\mathbf{A}})\neq \emptyset\), then \(\pi_{\mathbf{A}}(Fix(\phi^{\mathbf{A}})) \subseteq Fix(\phi)\).  
\end{lemma}

\begin{proof}
	Pick \(x\in W\) and let \(\xi\in M^\mathbf{A}_x\). We compute
$$
		\mathfrak  d_w (\phi^{\mathbf{A}}(\xi), \xi) = d_g(\phi(x), x) + \sup_{f\in \mathfrak {B}^1_k(M)}\Vert L_{\phi^{\mathbf{A}}(\xi)}(f) - L_{\xi}(f)\Vert_w = 
	  \sup_{f\in \mathfrak {B}^1_k(M)}\Vert L_{\xi}(f\circ \phi) - L_{\xi}(f)\Vert_w 
		$$ $$= \sup_{f\in \mathfrak {B}^1_k(M)}\Vert L_{\xi}(f\circ \phi - f)\Vert_w.
$$
	
	Since  the function \(z\mapsto (f\circ \phi - f)(z)\) is identically zero on \(W\) for each \(f\in \mathfrak {B}^1_k(M)\), by Morimoto's vanishing lemma, we have \(L_{\xi}(f\circ \phi - f) = 0\) for each \(f\in \mathfrak {B}^1_k(M)\). This implies that \(\mathfrak d_w (\phi^{\mathbf{A}}(\xi), \xi) = 0\) for all \(\xi\in M^\mathbf{A}_x\); that is, \(\phi^{\mathbf{A}}(\xi) = \xi\). Conversely, pick \(x\in \pi_{\mathbf{A}}(Fix(\phi^{\mathbf{A}}))\), and let \(\xi\in Fix(\phi^{\mathbf{A}})\cap \pi_{\mathbf{A}}^{-1}(\{x\})\). Then,
	\begin{align*}
		0 &= \mathfrak  d_w (\phi^{\mathbf{A}}(\xi), \xi) = d_g(\phi(x), x) + \sup_{f\in \mathfrak {B}^1_k(M)}\Vert L_{\phi^{\mathbf{A}}(\xi)}(f) - L_{\xi}(f)\Vert_w \\
		&= d_g(\phi(x), x) + \sup_{f\in \mathfrak {B}^1_k(M)}\Vert L_{\xi}(f) - L_{\xi}(f)\Vert_w = d_g(\phi(x), x).
	\end{align*}
	Thus, \(\phi(x) = x\).
\end{proof}
Lemma \ref{L-1-3} means that if an equilibrium state \(x\) is surrounded by other equilibrium states, then all infinitesimal perturbations of \(x\) are also stable under the lifted dynamics \(\phi^\mathbf{A}\). Also, if there exists even one stable infinitesimal perturbation, then its "shadow" on the base manifold \(M\) must be an equilibrium state. When the manifold \(M\)  in Lemma \ref{L-1-3} is a quotient space like \(S^1\), the fixed points are representatives of equivalence classes, but we still consider them as isolated fixed points within the topology of \(M\). \\
  
\textbf{Example 1: Rotation Map on \(S^1\) (Instability).} Let \(M = S^1\) (the circle), parameterized by \(\theta \in [0, 2\pi)\). Let \(\mathbf{A} = \mathbb{R}[\epsilon]/(\epsilon^2)\), so \(M^\mathbf{A} \cong TS^1\). Consider the rotation map \(\phi : S^1 \to S^1\) given by \(\phi(\theta) = \theta + \alpha\), where \(\alpha\) is a constant angle. An element \(\xi \in TS^1\) can be written as \(\xi = \theta + \epsilon v\), where \(\theta \in S^1\) and \(v \in \mathbb{R}\) represents a tangent vector at \(\theta\). The Weil lifting \(\phi^\mathbf{A} : TS^1 \to TS^1\) acts as follows:
$
\phi^\mathbf{A}(\xi) = \phi^\mathbf{A}(\theta + \epsilon v) = (\theta + \alpha) + \epsilon v.
$ The rotation map shifts the base point by \(\alpha\) but leaves the tangent vector unchanged. For \(\xi\) to be a fixed point of \(\phi^\mathbf{A}\), we need \(\phi^\mathbf{A}(\xi) = \xi\); that is:
$
(\theta + \alpha) + \epsilon v = \theta + \epsilon v.
$ This implies \(\theta + \alpha = \theta \pmod{2\pi}\), which means \(\alpha = 0 \pmod{2\pi}\).  In this case every single element is a fixed point (a "stable region"). However, if \(\alpha \neq 0 \pmod{2\pi}\), there are no fixed points for \(\phi^\mathbf{A}\). This is because even though the tangent vector is preserved, the base point \(\theta\) is always shifted by \(\alpha\), preventing \(\xi\) from being a fixed point. Therefore, when \(Fix(\phi)= \emptyset\) because it is assumed that \(\alpha \neq 0 \pmod{2\pi}\), then  by Lemma \ref{L-1-3}, 
\(Fix(\phi^\mathbf{A}) = \emptyset\). \\

\textbf{Example 2: Reflection on \(S^1\) (Stability).}
Let us consider \(M = S^1\) again, but this time, let \(\phi(\theta) = -\theta\).  The fixed points of \(\phi\) are \(\theta = 0\) and \(\theta = \pi\). Let \(\xi = \theta + \epsilon v \in TS^1\). Then the Weil lifting acts as: $
\phi^\mathbf{A}(\xi) = \phi^\mathbf{A}(\theta + \epsilon v) = -\theta - \epsilon v.
$ The reflection flips both the base point and the direction of the tangent vector. To find the fixed points of \(\phi^\mathbf{A}\), we need \(\phi^\mathbf{A}(\xi) = \xi\): $
-\theta - \epsilon v = \theta + \epsilon v.
$ This gives us two conditions: \(-\theta = \theta \pmod{2\pi}\), which means \(\theta = 0\) or \(\theta = \pi\), and 
\(-v = v\), which means \(v = 0\). (So, the only fixed points on \(TS^1\) are the zero tangent vectors at \(\theta = 0\) and \(\theta = \pi\).) Therefore, \(Fix(\phi^\mathbf{A}) = \{0 + \epsilon \cdot 0, \pi + \epsilon \cdot 0\}\). In this case, \(Fix(\phi^\mathbf{A})\) is non-empty. The fixed points of \(\phi^\mathbf{A}\) correspond to the fixed points of \(\phi\) with no infinitesimal motion (\(v = 0\)).  This reflects physical intuition: a stationary point (no velocity) is stable under reflection, but any infinitesimal motion would flip direction and destabilize. While \(\phi\) fixes two points on \(S^1\), \(\phi^\mathbf{A}\) fixes only their "trivial" tangent vectors. This aligns with Lemma \ref{L-1-3}: \(Fix(\phi^\mathbf{A})\) fibers over \(Fix(\phi)\), but the infinitesimal data must also remain unchanged.

\begin{center}
	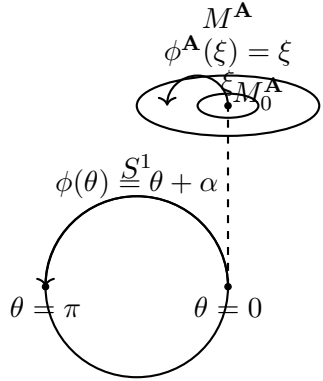
\begin{figure}[h!]
		\centering
		\begin{tikzpicture}[scale= 0.8]
			\draw[thick] (0,0) circle (1.5);
			\node at (0, 2) {\(S^1\)};
			\filldraw (1.5,0) circle (1.5pt) node[below] {\(\theta = 0\)};
			\filldraw (-1.5,0) circle (1.5pt) node[below] {\(\theta = \pi\)};
			
			\draw[->, thick] (1.5,0) arc (0:180:1.5);
			\node at (0, 1.7) {\(\phi(\theta) = \theta + \alpha\)};
			
			\draw[thick, dashed] (1.5,0) -- (1.5,3);
			\draw[thick] (1.5,3) ellipse (0.5 and 0.2);
			\node at (2, 3.2) {\(M^\mathbf{A}_0\)};
			\filldraw (1.5,3) circle (1.5pt) node[above] {\(\xi\)};
			
			\draw[->, thick] (1.5,3) arc (0:180:0.5);
			\node at (1.5, 3.9) {\(\phi^\mathbf{A}(\xi) = \xi\)};
			
			\draw[thick] (1.5,3) ellipse (1.5 and 0.5);
			\node at (1.5, 4.5) {\(M^\mathbf{A}\)};
		\end{tikzpicture}
		\caption{Fixed points of the rotation map \(\phi(\theta) = \theta + \alpha\) on \(S^1\) and its Weil lifting \(\phi^\mathbf{A}\) on \(M^\mathbf{A}\).}
		\label{fig:rotation-fixed-points}
	\end{figure}
\end{center}

\section{Topological invariants of \(M^\mathbf{A}\)}
The results on path lifting and connectedness suggest a strong relationship between the topology of \(M\) and the topology  of  \(M^\mathbf{A}\). The following theorems, using the powerful tools of the Leray spectral sequence and long exact sequence of homotopy groups, make this relationship precise.

\subsection{The group \(Diff^\infty(M^\mathbf A)\)}
Let \(Diff^\infty(M^\mathbf A)\) denote the group of all smooth diffeomorphisms from \(M^\mathbf A\) to \(M^\mathbf A\). We shall always assume that \(M\) is compact connected and equipped with a Riemannian metric \(g\).
The Weil functor \(F : Diff(M) \rightarrow Diff(M^\mathbf{A})\) maps diffeomorphisms on the base manifold \(M\) to diffeomorphisms on the Weil bundle \(M^\mathbf{A}\).   The Weil functor preserves the identity map and composition:
 \(id_M^\mathbf{A} = id_{M^\mathbf{A}}\) and 
	 \((\psi \circ \phi)^\mathbf{A} = \psi^\mathbf{A} \circ \phi^\mathbf{A}\). 
These properties ensure that \(F\) is indeed a functor from the category of smooth manifolds to itself.
The map \(F\) can be interpreted as a jet prolongation of diffeomorphisms, extending them to the Weil bundle.
\begin{center}
\begin{figure}[h!]
	\centering
	\begin{tikzpicture}[scale= 0.6]
		\node (M) at (0,0) {\(M\)};
		\node (MA) at (6,0) {\(M^\mathbf{A}\)};
		\node (N) at (0,-3) {\(N\)};
		\node (NA) at (6,-3) {\(N^\mathbf{A}\)};
		
		\draw[->, thick] (M) to node[midway, above, xshift= -0.3cm] {$\phi$} (N); 
		\draw[->, thick] (MA) to node[midway, above, xshift= 0.3cm] {$\phi^\mathbf{A}$} (NA); 
		\draw[->, thick] (M) to node[midway, left, yshift=0.2cm] {$F$} (MA); 
		\draw[->, thick] (N) to node[midway, right, yshift=-0.2cm] {$F$} (NA); 
		
		\node at (3, -1.5) {Commutative };
	\end{tikzpicture}
	\caption{The commutative diagram for \(F\) : Relationship between $Diff(M)$ and $ Diff(M^\mathbf{A})$.}
	\label{fig:weil-functor}
\end{figure}
\end{center}
\subsection*{The \(C^0\)-Topology} Let \(Homeo(M)\) denote the group of all homeomorphisms of \(M\) equipped with the compact-open topology. This is the metric topology induced by the distance : 
$
d_0^M(f,h) = \max(d_{C^0}(f,h),d_{C^0}(f^{-1},h^{-1})),
$
where
$
d_{C^0}(f,h) =\sup_{x\in M}d (h(x),f(x)).
$
On the space of all continuous paths \(\lambda:[0,1]\rightarrow Homeo(M)\) such that \(\lambda(0) = id_M\), we consider the \(C^0\)-topology as the metric topology induced by the metric:\\
$
\bar{d}_M(\lambda,\mu) = \max_{t\in [0,1]}d_0^M(\lambda(t),\mu(t)).
$

\subsection{A topology on \(Diff^\infty(M^\mathbf A)\)}

For each \(\tilde\phi \in Diff^\infty(M^\mathbf A)\) and for each \(\epsilon > 0\), we define a neighborhood \(U_{\tilde\phi,\epsilon}\) of \(\tilde\phi\) of order \(\epsilon\) as:
\[
U_{\tilde\phi,\epsilon} := \{\tilde\psi\in Diff^\infty(M^\mathbf A) \mid \max\{\mathfrak d_w (\tilde\phi(\xi), \tilde\psi(\xi)), \mathfrak d_w (\tilde\phi^{-1}(\xi), \tilde\psi^{-1}(\xi))\} < \epsilon, \forall\xi\in M^\mathbf A \}.
\]
We shall call the topology induced by the above family of open sets the "pointwise topology (p.w.t.)". \\ The following  result  makes $F$ a powerful tool for studying the relationship between the diffeomorphism group of $M$  and its Weil bundle $M^\mathbf A$. 
\begin{theorem}\label{Dyn-2}
	Let \(M\) be a smooth compact connected manifold and \(\mathbf A\) a Weil algebra. The map
	\[
	F : Diff^\infty(M)\longrightarrow Diff^\infty(M^\mathbf A), \qquad \phi\mapsto \phi^\mathbf A
	\]
	is continuous with respect to the \(C^0\)-topology on \(Diff^\infty(M)\) and the p.w.t. on \(Diff^\infty(M^\mathbf A)\).
\end{theorem}

\begin{proof}
	Let \(\{\phi_i\}\) be a sequence in \(Diff^\infty(M)\) that converges to \(\phi\in Diff^\infty(M)\). Pick \(\xi\in M^\mathbf A\). Using the linearity of the map \(f\mapsto L_{\xi}(f)\), we compute
	\begin{align*}
		\mathfrak d_w (\phi^{\mathbf{A}}_i(\xi), \phi^{\mathbf{A}}(\xi)) &= d_g(\phi_i(\pi_{\mathbf{A}}(\xi)), \phi(\pi_{\mathbf{A}}(\xi))) + \sup_{f\in \mathfrak {B}^1_k(M)}\Vert L_{\phi^{\mathbf{A}}_i(\xi)}(f) - L_{\phi^{\mathbf{A}}(\xi)}(f)\Vert_w \\
		&= d_g(\phi_i(\pi_{\mathbf{A}}(\xi)), \phi(\pi_{\mathbf{A}}(\xi))) + \sup_{f\in \mathfrak {B}^1_k(M)}\Vert L_{\xi}(f\circ \phi_i) - L_{\xi}(f\circ \phi)\Vert_w \\
		&= d_g(\phi_i(\pi_{\mathbf{A}}(\xi)), \phi(\pi_{\mathbf{A}}(\xi))) + \sup_{f\in \mathfrak {B}^1_k(M)}\Vert L_{\xi}(f\circ \phi_i - f\circ \phi)\Vert_w.
	\end{align*}
	
	On the other hand, the sequence of smooth functions $h_i := f\circ \phi_i$ uniformly converges to the smooth function  $ f\circ \phi$. Thus, using the continuity of the map $f\mapsto L_{\xi}(f)$ we derive that 
	\[
\lim_{i\to\infty}	\mathfrak d_w (\phi^{\mathbf{A}}_i(\xi), \phi^{\mathbf{A}}(\xi)) = \lim_{i\to\infty}d_g(\phi_i(\pi_{\mathbf{A}}(\xi)), \phi(\pi_{\mathbf{A}}(\xi))) + \lim_{i\to\infty}
\sup_{f\in \mathfrak {B}^1_k(M) }\Vert L_{\xi}(f\circ \phi_i - f\circ \phi)\Vert_w.
= 0,	\]
	for each \(\xi\in M^\mathbf A\).	We can use similar arguments to show that \(\mathfrak d_w ((\phi^{\mathbf{A}}_i)^{-1}(\xi), (\phi^{\mathbf{A}})^{-1}(\xi))\rightarrow 0\) as \(i\rightarrow \infty\). This is because for a bijective map \(\phi\), we have \((\phi^{\mathbf{A}})^{-1} = (\phi^{-1})^{\mathbf{A}}\). 
\end{proof}
\begin{table}[h!]
	\centering
	\begin{tabularx}{\textwidth}{@{}p{4cm} p{5cm} p{5cm}p{2cm} @{}} 
		\toprule
		\text{	Property of $Diff^k(M)$} 
		$(C^{0}-$Topology) & \text{Property of} $F(Diff^k(M))$ (Pointwise topology) & Conditions/Caveats & Result Proven? \\
		\midrule
		N/A  & \text{Continuous map} $Diff^k(M)$$\xrightarrow{F} $$F(Diff^k(M))$  & $F(\phi) = \phi^\mathbf{A}$ is the Weil lifting. & Yes \\
		Connected & Connected & & Implied by Continuity \\
		Path-Connected & Path-Connected & & Implied by Continuity \\
		Compact & Compact & & Implied by Continuity \\
		Locally Connected & Locally Connected & & Yes \\
		Local Contractibility & Local Contractibility & & Yes \\
		Transitive Group Action & Transitive Group Action & Generally NO. Strong conditions required. Consider local transitivity. & No \\
		\bottomrule
	\end{tabularx}
	\caption{Properties of $Diff^k(M)$ vs. $F(Diff^k(M))\subseteq Diff^k(M^\mathbf{A})$}
	\label{tab:DiffM_vs_DiffMA}
\end{table}
\newpage

The following results are  motivated by Lemma \ref{L-1} which  tells us that $	\pi_l(M^\mathbf{A}) \cong \pi_l(M) $ for $l =0, 1$.  
\begin{theorem}\label{Top-1}
	Let \( M \) be a smooth manifold. The homotopy groups satisfy:
	$$
	\pi_k(M^\mathbf{A}) \cong \pi_k(M) \quad \text{for all } k \geq 0.
	$$
\end{theorem}

\begin{proof}
	The proof relies on the long exact sequence of homotopy groups for a fibration. The bundle \(M^\mathbf{A} \to M\) is a fibration. The key observation is that the fiber, \(\mathbf{A}^{\dim M}\), is contractible (as a vector space). This contractibility implies that the homotopy groups of the fiber are trivial; that is, \(\pi_k(\mathbf{A}^{\dim M}) = 0\) for all \(k \geq 0\). Consider the long exact sequence of homotopy groups associated with the fibration \(M^\mathbf{A} \to M\):
	
	\begin{center}
		\begin{tikzpicture}[scale=0.7, node distance=1.5cm]
			\node (M1) at (-3,0) {$\cdots$};
			\node (Mk) at (0,0) {$\pi_{k+1}(M)$};
			\node (FAk) at (3,0) {$\pi_k(\mathbf{A}^{\dim M})$};
			\node (MAk) at (6,0) {$\pi_k(M^\mathbf{A})$};
			\node (Mk1) at (9,0) {$\pi_k(M)$};
			\node (FAk1) at (12,0) {$\pi_{k-1}(\mathbf{A}^{\dim M})$};
			\node (M2) at (15,0) {$\cdots$};
			
			\draw[->] (M1) -- (Mk);
			\draw[->] (Mk) -- (FAk);
			\draw[->] (FAk) -- (MAk);
			\draw[->] (MAk) -- (Mk1);
			\draw[->] (Mk1) -- (FAk1);
			\draw[->] (FAk1) -- (M2);
		\end{tikzpicture}
	\end{center}
	
	This is the long exact sequence of homotopy groups associated with the fibration \(M^\mathbf{A} \to M\), where the fiber is \(\mathbf{A}^{\dim M}\). Since \(\pi_k(\mathbf{A}^{\dim M}) = 0\) for all \(k\), the long exact sequence simplifies significantly. Thus, segments of the sequence become:
	$
	\dots \to \pi_{k+1}(M) \to 0 \to \pi_k(M^\mathbf{A}) \to \pi_k(M) \to 0 \to \dots
	$
	This implies that the maps \(\pi_k(M^\mathbf{A}) \to \pi_k(M)\) are isomorphisms for all \(k \geq 0\). Consequently, we can invoke the long exact sequence of homotopy groups for a fibration to deduce that \(\pi_k(M^\mathbf{A}) \cong \pi_k(M)\) for all \(k \geq 0\) \cite{Hatcher2002}.
\end{proof}
\begin{remark}\label{Leray}
	The Leray spectral sequence is a powerful tool for computing the cohomology of a space by relating it to the cohomology of another space through a map between them \cite{Getzler1994}. 
	The Leray spectral sequence for the projection map \(\pi_M: M^\mathbf{A} \to M\) has the following \(E_2\) page:
	$
	E_2^{p,q} = H^p(M; R^q \pi_M* \mathbb{R}),
	$ where \(R^q \pi_M* \mathbb{R}\) is the \(q\)-th higher direct image sheaf of the constant sheaf \(\mathbb{R}\) on \(M^\mathbf{A}\) under the map \(p\).  Specifically, for each point \(x \in M\), \(R^q \pi_M* \mathbb{R}_x\) is the cohomology \(H^q(\pi_M^{-1}(x); \mathbb{R})\) of the fiber \(\pi_M^{-1}(x)\) over \(x\). 
	Since the fibers of \(\pi_M\) are contractible (each is isomorphic to $ \mathbf{A}^{\dim M}$ ), we have:
	$
	H^q(\pi_M^{-1}(x); \mathbb{R}) =
	\begin{cases}
		\mathbb{R} & \text{if } q = 0 \\
		0 & \text{if } q > 0.
	\end{cases}
	$\\
	Therefore, the higher direct image sheaves simplify to:
	$
	R^q \pi_M* \mathbb{R} =
	\begin{cases}
		\mathbb{R} & \text{if } q = 0 \\
		0 & \text{if } q > 0.
	\end{cases}
	$
	
	Consequently, the \(E_2\) page becomes: $
	E_2^{p,q} =
	\begin{cases}
		H^p(M; \mathbb{R}) & \text{if } q = 0 \\
		0 & \text{if } q > 0.
	\end{cases}
	$\\
	
	This indicates that the Leray spectral sequence collapses at the \(E_2\) page, meaning \(E_2 = E_\infty\).
\end{remark}

\begin{theorem}\label{Top-2}
		Let \(M\) be a smooth compact connected manifold and \(\mathbf A\) a Weil algebra.
	For the bundle \( M^\mathbf{A} \to M \), the spectral sequence collapses at \( E_2 \):
	$
	E_2^{p,q} = H^p(M, H^q(\mathbf{A}^{\dim M})) \implies H^{p+q}(M^\mathbf{A}).
	$
	Since \( H^q(\mathbf{A}^{\dim M}) = 0 \) for \( q > 0 \), we recover:
	$
	H^*(M^\mathbf{A}) \cong H^*(M).
	$
\end{theorem}

	\begin{proof}
		The Weil bundle \( M^\mathbf{A} \) is a fiber bundle over \( M \) with fiber \( F \cong \mathbf{A}^{\dim M} \). As \( \mathbf{A} \) is a finite-dimensional real vector space, the fiber \( F \) is contractible. This implies trivial cohomology groups:
		\[
		H^q(F) \cong 
		\begin{cases} 
			\mathbb{R} & \text{if } q = 0, \\
			0 & \text{if } q > 0.
		\end{cases}
		\]
		
The Leray spectral sequence for \(M^\mathbf{A} \to M\) has \(E_2\) page:
		$
		E_2^{p,q} \cong H^p\big(M, H^q(F)\big) \implies H^{p+q}(M^\mathbf{A}).
		$
		For \( q > 0 \), all terms \( E_2^{p,q} \) vanish, leaving only the bottom row:
		$
		E_2^{p,0} \cong H^p(M).
		$ Since all differentials \( d_r \) for \( r \geq 2 \) must map between trivial groups for \( q > 0 \), the spectral sequence collapses immediately:
		$
		E_2^{p,q} \cong E_\infty^{p,q}.
		$ The associated graded algebra satisfies:
		$
		 H^k(M^\mathbf{A}) \cong \bigoplus_{p + q = k} E_\infty^{p,q} \cong H^k(M),
		$
		and since the filtration is trivial, we conclude:
		$
		H^*(M^\mathbf{A}) \cong H^*(M). \qedhere
		$
	\end{proof}

\subsection{Cohomology ilustrations:}

\begin{itemize}
\item Tangent Bundle of $S^2$. 
\text{Base}: $M = S^2$. 
\text{Weil Algebra}: $\mathbf{A} = \mathbb{R}[\epsilon]/(\epsilon^2)$ .\\
\text{Weil Bundle}: $M^\mathbf{A} = T S^2$ (tangent bundle). 
\text{Fiber}: $\mathcal{A}^{\dim M} = \mathbb{R}^2$ (contractible).\\
\text{Cohomology}:
$
H^k(T S^2) \cong 
\begin{cases} 
	\mathbb{R} & k = 0, 2 \\
	0 & \text{otherwise}.
\end{cases}
$\\
\text{Spectral sequence collapse}:
$
E_2^{p,q} = 
\begin{cases} 
	H^p(S^2) & q = 0 \\
	0 & q > 0 
\end{cases} \implies H^*(T S^2) \cong H^*(S^2).
$
\item Second-Order Jet Bundle of $T^2$. 
\text{Base}: $M = T^2$. 
\text{Weil Algebra}: $\mathbf{A} = \mathbb{R}[\epsilon_1, \epsilon_2]/(\epsilon_1^2, \epsilon_2^2, \epsilon_1\epsilon_2)$. \\
\text{Weil Bundle}: $M^\mathbf{A} = J^2(T^2)$ ($2-$jet bundle). 
\text{Fiber}: $\mathbf{A}^{\dim M} = \mathbb{R}^4$ (contractible). \\
\text{Cohomology}:
$
H^k(J^2(T^2)) \cong 
\begin{cases} 
	\mathbb{R} & k = 0 \\
	\mathbb{R}^2 & k = 1 \\
	\mathbb{R} & k = 2 \\
	0 & k > 2.
\end{cases}
$\\
\text{Spectral esquence collapse}:
$
E_2^{p,q} = 
\begin{cases} 
	H^p(T^2) & q = 0 \\
	0 & q > 0 
\end{cases} \implies H^*(J^2(T^2)) \cong H^*(T^2).
$
\item Trivial Weil Algebra over $\mathbb{RP}^3$. 
\text{Base}: $M = \mathbb{RP}^3$. 
\text{Weil Algebra}: $\mathbf{A} = \mathbb{R}$. \\
\text{Weil Bundle}: $M^\mathbf{A} \cong M \times \{\text{pt}\} \cong M$. 
\text{Fiber}: $\mathbf{A}^{\dim M} = \{\text{pt}\}$. \\
\text{Cohomology}:
$
H^k(M^\mathbf{A}) \cong 
\begin{cases} 
	\mathbb{R} & k = 0, 3 \\
	0 & \text{otherwise}.
\end{cases}
$\\
\item Product Weil Algebra over $\mathbb{CP}^1$.  
\text{Base}: $M = \mathbb{CP}^1$ .
\text{Weil Algebra}: $\mathbf{A} = \mathbb{R}[\epsilon]/(\epsilon^2) \otimes \mathbb{R}[\eta]/(\eta^2)$. \\
\text{Weil Bundle}: $M^\mathbf{A}$ (encoding two tangent directions). 
\text{Fiber}: $\mathbf{A}^{\dim M} = \mathbb{R}^4$ (contractible).  
\text{Cohomology}:
$
H^k(M^\mathbf{A}) \cong 
\begin{cases} 
	\mathbb{R} & k = 0 \\
	0 & k = 1 \\
	\mathbb{R} & k = 2 \\
	0 & k > 2.
\end{cases}
$\\
\text{Spectral sequence collapse}:
$
E_2^{p,q} = 
\begin{cases} 
	H^p(\mathbb{CP}^1) & q = 0 \\
	0 & q > 0 
\end{cases} \implies H^*(M^\mathbf{A}) \cong H^*(\mathbb{CP}^1).
$
\end{itemize}

\section*{Open problems}
\begin{enumerate}
	\item \text{Geodesic Completeness}: Does \((M^\mathbf{A}, \mathfrak{d}_w)\) admit a Hopf–Rinow theorem? Establishing geodesic completeness would guarantee the existence of minimizing geodesics between any two points, crucial for optimization and path planning on Weil bundles.
	\item \text{Ergodicity}: If \(\phi: M \to M\) is ergodic, is \(\phi^\mathbf{A}: M^\mathbf{A} \to M^\mathbf{A}\) also ergodic? Understanding the ergodicity of Weil liftings has implications for the long-term behavior of dynamical systems with infinitesimal structure.
	\item \text{Machine Learning}: Can \(M^\mathbf{A}\) serve as a smooth latent space for neural networks informed by partial differential equations (PDEs)?
		\item \text{	Sheaf Cohomology}: Can infinitesimal data in \(M^\mathbf{A}\)
	detect topological obstructions ?
\end{enumerate}

\begin{center}
\textbf{	Acknowledgments}
\end{center}
\begin{center}
We dedicate this work to the memory of Professor Okassa, whose commitment and dedication to mathematics, especially his foundational work on the theory of infinitely near points, inspired us.

\end{center}
This research was initiated during a workshop organized by AFRIMath-CRNS in Porto-Novo, Benin (2023).  We gratefully acknowledge their support.\\

\textbf{No funding was received for the accomplishment of this work.}\\
\textbf{The authors Claim No Conflict of Interest.}

\end{document}